\documentclass[pdflatex,sn-mathphys-num]{sn-jnl}

\usepackage{graphicx}%
\usepackage[T1]{fontenc}
\usepackage{multirow}%
\usepackage{amsmath,amssymb,amsfonts}%
\usepackage{amsthm}%
\usepackage{mathrsfs}%
\usepackage[title]{appendix}%
\usepackage{xcolor}%
\usepackage{textcomp}%
\usepackage{manyfoot}%
\usepackage{booktabs}%
\usepackage{algorithm}%
\usepackage{algorithmicx}%
\usepackage{algpseudocode}%
\usepackage{listings}%

\theoremstyle{thmstyleone}%
\newtheorem{theorem}{Theorem}
\newtheorem{proposition}[theorem]{Proposition}%

\theoremstyle{thmstyletwo}%

\theoremstyle{thmstylethree}%
\newtheorem{definition}{Definition}%

\begin{document}

\title[A complete classification of endomorphisms of Kiselman's semigroup]{A complete classification of endomorphisms of Kiselman's semigroup}


\author*[1]{\fnm{Luka} \sur{Andren\v sek}}\email{andrensek.luka@gmail.com}

\affil*[1]{\orgdiv{Department of Mathematics}, \orgname{Faculty of Mathematics and Physics, University of Ljubljana}, \orgaddress{\street{Jadranska ulica 19}, \city{Ljubljana}, \postcode{SI-1000},  \country{Slovenia}}}


\abstract{Kiselman's semigroup $K_n$ was studied by Kudryavtseva and Mazorchuk, who posed the question of whether it is possible to classify all endomorphisms of $K_n$. In this paper, we provide a complete classification of endomorphisms of $K_n$ and present a Boolean matrix monoid that is isomorphic to $\text{End}(K_n)$.}




\maketitle

\section{Introduction and the main result}\label{sec1}

Kiselman's semigroup was first introduced by Kiselman in the context of convexity theory in \cite{Kiselman}. A generalisation was later introduced and thoroughly studied by 
Kudryavtseva and Mazorchuk in \cite{GK}. The semigroup arises naturally in various settings such as graph dynamics as discussed by Collina and D’Andrea in \cite{Graphs}.

In \cite{ganyushkin11}, Ganyushkin and Mazorchuk introduced Kiselman quotients of 0-Hecke monoids associated with simply laced Dynkin diagrams, and
more generally, Hecke--Kiselman monoids associated with finite simple digraphs. Further properties and identities of Hecke--Kiselman monoids were studied by Ashikhmin, Volkov and Zhang in \cite{Ashikhmin15},  
irreducible representations of Hecke--Kiselman monoids were studied by Wiertel in \cite{wiertel22},
and the word problem for some particular Hecke--Kiselman monoids was studied by Lebed in \cite{Lebed2025}. 
Furthermore, the finiteness of Hecke--Kiselman monoids was investigated by Aragona and D’Andrea in 
\cite{aragona13}, while D’Andrea and Stella showed in \cite{dandrea23} that the cardinality of Kiselman's semigroups grows double-exponentially. 
 
Kiselman's semigroup is also related to 0-Hecke algebras \cite{norton79}. Properties and structure of Hecke--Kiselman algebras were studied by various authors in \cite{okninski20, okninski21, wiertel23, Mecel2019}.  

In order to formulate the main result of our paper, we define the following notions.
Let $M \in \mathbb{R}^{m \times n}$ be a real $m \times n$ matrix. Then for $1 \leq i \leq m$ and $1 \leq j \leq n$ we denote by $M_{i, j}$ the entry of the matrix $M$ in the $i$-th row and the $j$-th column. 
\begin{definition}
   Let $A \in \mathbb{R}^{2 \times 2}$ and $M \in \mathbb{R}^{n \times n}$ be real matrices with $n \geq 2$. We say that $A$ is a \emph{$2 \times 2$ submatrix} of 
   $M$ if there exist columns $i, j \in \{1, 2, \dots, n\}$ and rows $x, y \in \{1, 2, \dots, n\}$ with $i < j$ and $x < y$ such that 
   \[
      A = 
      \begin{pmatrix}
         M_{x, i} & M_{x, j} \\
         M_{y, i} & M_{y, j}
      \end{pmatrix}.
   \]
\end{definition}

Now we define the set of Boolean matrices that avoid a certain submatrix.

\begin{definition}\label{Dn}
   Let $n \geq 2$ be an integer, and let
   \begin{equation}\label{pattern}
      P = 
      \begin{pmatrix}
         0 & 1 \\
         1 & 0
      \end{pmatrix}.
   \end{equation}
   We define the set $D_n$ as 
   \[
      D_n = \{ M \in \{0, 1\}^{n \times n} \mid P \text{ is not a } 2 \times 2 \text{ submatrix of } M \}.
   \]
\end{definition}

On the set $D_n$, we define a binary operation $\cdot$ as follows. Let $A, B \in D_n$. Then $C = A \cdot B$, with
\begin{align*}
   C_{i, j} = \bigvee_{k=1}^n (A_{i, k} \land B_{k, j}).
\end{align*}
Here, $\lor$ denotes the Boolean join and $\land$ denotes the Boolean meet on $\{0, 1\}$. One can easily verify that $(D_n, \cdot)$ is a monoid, and its unit element is the identity matrix. 

Let $\mathcal{A} = \{a_1, a_2, \dots , a_n\}$ be a finite alphabet. Define Kiselman's semigroup $K_n$ by
\begin{align*}
   K_n = \langle a_1, a_2, \dots, a_n \mid a_i^2 = a_i, \; a_i a_j a_i = a_j a_i a_j = a_j a_i, \; 1 \leq i < j \leq n \rangle .
\end{align*}

In \cite{GK}, the authors showed in Proposition 20 that the only automorphism of $K_n$ is the identity, and that the only antiautomorphism is the unique extension to an antiautomorphism of the map 
$a_i \mapsto a_{n-i+1}$. They leave open the question (Question~21) concerning the classification of endomorphisms of $K_n$, which has not yet been answered.

The main result of this paper is the following theorem.
\begin{theorem}\label{thm:main}
   Let $n \geq 2$ be an integer. Then $\operatorname{End}(K_n) \cong (D_n, \cdot)$.
\end{theorem}

\section{Proof of the main result}
\subsection{Properties of idempotents of $K_n$}

Let $X \subseteq \{1, 2, \dots , n\}$. If $X = \emptyset$ define $e_X = e$, the unit element in $K_n$. Otherwise, write $X = \{i_1, i_2, \dots , i_k \mid i_1 > i_2 > \dots > i_k\}$. Then set
$e_X = a_{i_1} a_{i_2} \dots a_{i_k}$. Now define $E_n = \{e_X \mid X \subseteq \{1, 2, \dots, n\}\}$. In Proposition 11 of \cite{GK}, it is shown that every element of $E_n$ is an idempotent, that the elements of $E_n$ are pairwise distinct, and that every idempotent of $K_n$ is of the form $e_X$ for some $X \subseteq \{1, 2, \dots , n\}$. Consequently, the number of idempotents in $K_n$ is precisely $2^n$.

In \cite[Proposition 14]{GK}, the following proposition was proved. 

\begin{proposition}\label{idem:eq}
   Let $X, Y \subseteq \{1, 2, \dots, n\}$. Then the following conditions are equivalent:
   \begin{enumerate}
      \item $e_X e_Y$ is an idempotent.
      \item $e_X e_Y = e_{X \cup Y}$.
      \item For every $x \in X \setminus Y$ and every $y \in Y \setminus X$ we have $x > y$.
   \end{enumerate}
\end{proposition}

Define the \emph{content} map $c : K_n \rightarrow 2^{\{1, 2, \dots, n\}}$, where $c(a)$ is the set of all $i$'s such that $a_i$ appears in $a$. In \cite{GK}, it is shown that this map is well defined.
Furthermore, in Lemma 10 of \cite{GK} it is shown that the following holds.

\begin{proposition}
  The map $c$ is a semigroup epimorphism from $K_n$ onto the semigroup $(2^{\{1, 2, \dots, n\}}, \cup)$.
\end{proposition} 

We now prove the following proposition.
\begin{proposition}\label{gate}
   Let $X, Y \subseteq \{1, 2, \dots , n\}$. The following two conditions are equivalent:
   \begin{enumerate}
      \item For every $x \in X \setminus Y$ and every $y \in Y \setminus X$ we have $x > y$.
      \item $e_X e_Y$ is an idempotent and $e_X e_Y e_X = e_Y e_X e_Y = e_X e_Y$.
   \end{enumerate}
\end{proposition}

\begin{proof}
   The implication \((2) \Rightarrow (1)\) is immediate from Proposition~\ref{idem:eq}.

   By Proposition~\ref{idem:eq}, to prove \((1) \Rightarrow (2)\), it suffices to show the equalities:
   \[
   e_X e_Y e_X = e_Y e_X e_Y = e_X e_Y.
   \]

   Since $e_X e_Y$ is an idempotent, it is equal to $e_{X \cup Y}$. Observe that condition $(3)$ in Proposition~\ref{idem:eq} holds for the pairs of sets $(X \cup Y, X)$ and $(Y, X \cup Y)$, respectively. It follows that 
   \begin{align*}
      e_X e_Y e_X = (e_X e_Y) e_X = e_{X \cup Y} e_X = e_{(X \cup Y) \cup X} = e_{X \cup Y} = e_X e_Y,
   \end{align*}
   and similarly
   \begin{align*}
      e_Y e_X e_Y = e_Y (e_X e_Y ) = e_Y e_{X \cup Y } = e_{Y \cup (X \cup Y)} = e_{X \cup Y} = e_X e_Y.
   \end{align*}
   This concludes the proof.
\end{proof}

\subsection{Isomorphism of $\operatorname{End}(K_n)$ with the monoid of monotone sequences of sets $M_n$}

In this section, we construct an isomorphism from $\text{End}(K_n)$ to the set of sequences of sets satisfying the \textit{monotonicity} condition, which we now define. 

\begin{definition}
   Let $\mathbf{X} = (X_1, X_2, \dots, X_n)$ be a sequence of sets, where $X_i \subseteq \{1, 2, \dots, n\}$ for $1 \leq i \leq n$. We say that $\mathbf{X}$ is \emph{monotone} if
   \begin{align*}
      &\text{for all } i, j \in \{1, 2, \dots, n\} \text{ with } j > i, \\
      &\text{and for all } x \in X_j \setminus X_i \text{ and } y \in X_i \setminus X_j, \text{ we have } x > y.
   \end{align*}
   We denote by $M_n$ the set
   \[
   M_n = \{(X_1, X_2, \dots, X_n) \mid (X_1, X_2, \dots, X_n) \text{ is monotone}\}.
   \]
\end{definition}

On the set $M_n$, we define a binary operation $\ast$ as follows.
Let $\mathbf{X} = (X_1, X_2, \dots, X_n)$ and $\mathbf{Y} = (Y_1, Y_2, \dots, Y_n)$ be elements of $M_n$. 
We define $\mathbf{Z} = \mathbf{X} \ast \mathbf{Y}$, where $\mathbf{Z} = (Z_1, Z_2, \dots, Z_n)$ is given by
\begin{align*}
   Z_i = \bigcup\limits_{j \in Y_i} X_j
\end{align*}
for $1 \leq i \leq n$. It is easy to verify that $(M_n, \ast)$ is a monoid, with its unit element given by the sequence $(\{1\}, \{2\}, \dots,\{n\})$.

\begin{theorem}\label{thm:Phi}
   The map $\Phi : \operatorname{End}(K_n) \rightarrow M_n$, defined by
   \begin{align*}
      \Phi(\varphi) = \big( c(\varphi(a_1)), c(\varphi(a_2)), \dots, c(\varphi(a_n)) \big)
   \end{align*}
   is an isomorphism of monoids $\operatorname{End}(K_n)$ and $(M_n, \ast)$.
\end{theorem}

\begin{proof}
   First, we show that for every $\varphi \in \operatorname{End}(K_n)$, $\Phi(\varphi)$ lies in $M_n$. Let $\varphi \in \operatorname{End}(K_n)$. 
   Since $\varphi$ is a homomorphism, it maps idempotents to idempotents. In particular, $\varphi(a_i)$ is an idempotent for 
   $1 \leq i \leq n$. As every idempotent in $K_n$ is of the form $e_X$, it follows that there exists a sequence $(X_1, X_2, \dots, X_n)$ of subsets of $\{1, 2, \dots, n\}$ such that
   \begin{align*}
      \varphi(a_i) = e_{X_i}
   \end{align*}
   for $1 \leq i \leq n$. Thus we have $\Phi(\varphi) = (X_1, X_2, \dots, X_n)$. Now observe that $a_j a_i$ is an idempotent, whenever $j > i$, and in this case we also have $a_ja_ia_j = a_i a_j a_i = a_j a_i$. Applying $\varphi$ to these equalities and using the fact that $\varphi$ preserves idempotents yields:
   \begin{align*}
      & e_{X_j} e_{X_i} \text{ is an idempotent and} \\
      & e_{X_j} e_{X_i} e_{X_j} = e_{X_i} e_{X_j} e_{X_i} = e_{X_j} e_{X_i}
   \end{align*}
   whenever $j > i$. By Proposition~\ref{gate} it follows that for every $j>i$, the condition 
   \begin{align*}
      \text{for every } x \in X_j \setminus X_i\; \text{and every } y \in X_i \setminus X_j\; \text{we have } x > y
   \end{align*}
   must be satisfied. This implies the monotonicity of $(X_1, X_2, \dots, X_n)$ and hence $\Phi(\varphi) \in M_n$.

   The fact that $\Phi$ is injective follows quickly, since for any endomorphism of $K_n$, the images of the generators are of the form $e_{X_i}$ for some $X_i \subseteq \{1, 2, \dots, n\}$.
   Observe that an idempotent in $K_n$ is uniquely determined by its content. Since an endomorphism is uniquely determined by the images of the generators, it follows that an endomorphism of $K_n$ is uniquely determined by the contents of the images of the generators, from which the injectivity of $\Phi$ follows.

   Now we prove surjectivity. Let $\mathbf{X} = (X_1, X_2, \dots, X_n) \in M_n$. Define the map $\psi : \{a_1, a_2, \dots, a_n\} \rightarrow K_n$ by 
   \begin{align*}
      \psi(a_i) = e_{X_i}.
   \end{align*}
   Set $b_i = \psi(a_i)$ for $1 \leq i \leq n$. Clearly, each $b_i$ is an idempotent. Furthermore, let $j > i$. Then, by Proposition~\ref{gate} and the monotonicity condition imposed on $(X_1, X_2, \dots, X_n)$ it must follow that
   \begin{align*}
      b_j b_i b_j = b_i b_j b_i = b_j b_i.
   \end{align*}
   Thus, the elements $\{b_1, b_2, \dots, b_n\}$ satisfy the defining relations of $K_n$, so we can uniquely extend $\psi$ to an endomorphism of $K_n$. By construction, we have that $\Phi(\psi) = \mathbf{X}$, which proves surjectivity. 


   Now we prove that $\Phi$ is a homomorphism. Let $\psi, \varphi \in \operatorname{End}(K_n)$. Let
   $\psi(a_j) = e_{Y_j}$ and $\varphi(a_j) = e_{X_j}$ for $1 \leq j \leq n$. Also write $e_{X_i} = a_{m_l} a_{m_{l-1}} \dots a_{m_1}$, where $X_i = \{m_l, m_{l-1}, \dots, m_1\}$ and $m_l > m_{l-1} > \dots > m_1$. We compute
   \begin{align*}
      \hspace{-1.3cm} (\psi \circ \varphi)(a_i) &= \psi(e_{X_i}) \\
      &= \psi(a_{m_l} a_{m_{l-1}} \dots a_{m_2} a_{m_1})  \\
      &= \psi(a_{m_l}) \psi(a_{m_{l-1}}) \dots \psi(a_{m_2}) \psi(a_{m_1}) \\
      &= e_{Y_{m_l}} e_{Y_{m_{l-1}}} \dots e_{Y_{m_2}} e_{Y_{m_1}} \\
      &= e_{\cup_{1 \leq k \leq l} Y_{m_k}}.
   \end{align*}
   The last equality holds since $(\psi \circ \varphi)(a_i)$ is an idempotent, every idempotent in $K_n$ is of the form $e_X$, and the content of an element in $K_n$ is preserved under its defining relations, since the content map $c$ is well defined. Now, if we denote by
   $(\Phi(\psi \circ \varphi))_i$ the $i$-th component of the sequence of sets $\Phi(\psi \circ \varphi)$, we get
   
   \begin{align*}
      \hspace{-6cm} \big(\Phi(\psi \circ \varphi)\big)_i &= c((\psi \circ \varphi)(a_i)) \\
      &= c(e_{\cup_{1 \leq k \leq l} Y_{m_k}}) \\
      &= \bigcup\limits_{k =1}^{l} Y_{m_k} \\
      &= \bigcup\limits_{j \in X_i} Y_{j},
   \end{align*}
   since $X_i = \{m_l, m_{l-1}, \dots, m_1\}$.

   On the other hand, one has
   \begin{align*}
      \big(\Phi(\psi) \ast \Phi(\varphi)\big)_i &= \big((c(\psi(a_1)), c(\psi(a_2)), \dots, c(\psi(a_n))) \ast (c(\varphi(a_1)), c(\varphi(a_2)), \dots, c(\varphi(a_n)))\big)_i \\
      &= \big((Y_1, Y_2, \dots, Y_n) \ast (X_1, X_2, \dots, X_n)\big)_i \\
      &= \bigcup\limits_{j \in X_i} Y_j.
   \end{align*}
   We have shown that $\big(\Phi(\psi \circ \varphi)\big)_i = \big(\Phi(\psi) \ast \Phi(\varphi)\big)_i$,
   and since $i$ was chosen arbitrarily, it follows that
   \begin{align*}
      \Phi(\psi \circ \varphi) = \Phi(\psi) \ast \Phi(\varphi).
   \end{align*}
   Since the identity map is the unit of $\operatorname{End}(K_n)$, and because of
   \begin{align*}
      \Phi(\operatorname{id}_{K_n}) = (\{1\}, \{2\}, \dots, \{n\}),
   \end{align*}
   we conclude that $\Phi$ is an isomorphism of monoids. This completes the proof. 

\end{proof}

\subsection{Isomorphism of $M_n$ with $D_n$}

One can conclude that an $n \times n$ Boolean matrix $M$ is an element of $D_n$, as defined in Definition~\ref{Dn},  for $n \geq 2$, if and only if there do not exist columns 
$i, j \in \{1, 2, \dots, n\}$ and rows $x, y \in \{1, 2, \dots, n\}$ with $i < j$, and $x < y$, such that
\begin{align*}
   &M_{x, i} = 0, \quad M_{x, j} = 1, \\
   &M_{y, i} = 1, \quad M_{y, j} = 0.
\end{align*}
Also observe that if $A$ and $B$ are $n \times n$ Boolean matrices, and if for all $i, j \in \{1, 2, \dots, n\}$ we have that 
$A_{i, j} = 1$ if and only if $B_{i, j} = 1$, then $A = B$.  

Let $(e_1, e_2, \dots, e_n)$ be the standard orthonormal basis of $\mathbb{R}^n$. Define the map 
$\chi : 2^{\{1, 2, \dots, n\}} \rightarrow \{0, 1\}^n$ by
\begin{align*}
   \chi(X) = \sum_{i \in X} e_i.
\end{align*}
If further $v \in \mathbb{R}^n$ is a vector, then for $1 \leq i \leq n$ we denote by $v_i$ the $i$-th element of the vector $v$. 

\begin{theorem}\label{thm:Psi}
   The map $\Psi : M_n \rightarrow D_n$, defined by
   \[
      \Psi(X_1, X_2, \dots, X_n) = (\chi(X_1), \chi(X_2), \dots, \chi(X_n))
   \]
   where the right-hand side is interpreted as a matrix whose $i$-th column is $\chi(X_i)$, is an isomorphism of monoids $(M_n, \ast)$ and $(D_n, \cdot)$.
\end{theorem}

\begin{proof}
   We first show that $\Psi$ maps into $D_n$. Let $\mathbf{X} = (X_1, X_2, \dots, X_n) \in M_n$ and denote $M = \Psi(\mathbf{X})$.
   Observe that $M_{x, i} = 1$ if and only if $(\chi(X_i))_x = 1$, which is equivalent to $x \in X_i$.
   Now assume $i < j$ and assume that the following equalities hold
   \begin{align*}
      &M_{x, i} = 0, \quad M_{x, j} = 1, \\
      &M_{y, i} = 1, \quad M_{y, j} = 0.
   \end{align*}
   We aim to show that $x < y$ is impossible. 
   The upper equalities are equivalent to 
   \begin{align*}
      &x \notin X_i, \quad x \in X_j, \\
      &y \in X_i, \quad y \notin X_j, 
   \end{align*}
   which is further equivalent to 
   \[ 
      x \in X_j \setminus X_i \quad \text{and} \quad y \in X_i \setminus X_j. 
   \]
   We apply monotonicity of $(X_1, X_2, \dots, X_n)$ and get $x>y$. Hence, $\Psi(\mathbf{X}) \in D_n$.

   The map $\Psi$ is obviously injective.

   We now prove surjectivity. Let $M \in D_n$. Define $X_i = \{x \in \{1, 2, \dots, n\} \mid M_{x, i} = 1\}$. First, we check that $(X_1, X_2, \dots, X_n)$ is monotone.
   Let $i < j$ and suppose $x \in X_j \setminus X_i$ and $y \in X_i \setminus X_j$. Then we have
   \begin{align*}
      M_{x, i} = 0, \quad M_{x, j} = 1, \\
      M_{y, i} = 1, \quad M_{y, j} = 0.
   \end{align*}
   Since $M \in D_n$, it must follow that $x>y$, and hence $(X_1, X_2, \dots, X_n)$ is monotone.
   Furthermore, for $x, i \in \{1, 2, \dots, n\}$ we have that $(\Psi(X_1, X_2, \dots, X_n))_{x, i} = 1$ if and only if $x \in X_i$, which is equivalent to  $M_{x, i} = 1$ by the definition of $X_i$.
   Hence
   \[
   \Psi(X_1, X_2, \dots, X_n) = M,
   \]
   and surjectivity is proven.

   Now we show that $\Psi$ is a homomorphism. Let $\mathbf{X}, \mathbf{Y} \in M_n$, where
   $\mathbf{X} = (X_1, X_2, \dots, X_n)$ and $\mathbf{Y} = (Y_1, Y_2, \dots, Y_n)$. We compute
   \begin{align*}
      \Psi(\mathbf{X} \ast \mathbf{Y}) = \Psi(\bigcup_{j \in Y_1} X_j, \bigcup_{j \in Y_2} X_j, \dots, \bigcup_{j \in Y_n} X_j).
   \end{align*}
   It follows that for $x, i \in \{1, 2, \dots, n\}$ we have 
   \begin{align*}
      \Psi(\mathbf{X} \ast \mathbf{Y})_{x, i} = 1 &\Leftrightarrow x \in \bigcup_{j\in Y_i}X_j \\
      &\Leftrightarrow \text{there exists } k \in Y_i \text{ such that } x \in X_k.
   \end{align*}
   On the other hand,
   \begin{align*}
      (\Psi(\mathbf{X}) \cdot \Psi(\mathbf{Y}))_{x, i} = \bigvee_{k=1}^n (\Psi(\mathbf{X})_{x, k} \land \Psi(\mathbf{Y})_{k, i}).
   \end{align*}
   Hence, we have
   \begin{align*}
      (\Psi(\mathbf{X}) \cdot \Psi(\mathbf{Y}))_{x, i} = 1 &\Leftrightarrow \text{there exists } k \in \{1, 2, \dots, n\} \text{ such that } \Psi(\mathbf{X})_{x, k} = \Psi(\mathbf{Y})_{k, i} = 1 \\
      &\Leftrightarrow \text{there exists } k \in \{1, 2, \dots, n\} \text{ such that } x \in X_k \text{ and } k \in Y_i \\
      &\Leftrightarrow \text{there exists } k \in Y_i \text{ such that } x \in X_k.
   \end{align*}
   We thus obtain
   \[
      \Psi(\mathbf{X} \ast \mathbf{Y})_{x, i} = (\Psi(\mathbf{X}) \cdot \Psi(\mathbf{Y}))_{x, i}.
   \]
   Since $x$ and $i$ were chosen arbitrarily, we conclude that
   \begin{align*}
      \Psi(\mathbf{X} \ast \mathbf{Y}) = \Psi(\mathbf{X}) \cdot \Psi(\mathbf{Y}).
   \end{align*}
   Furthermore, since $\Psi(\{1\}, \{2\}, \dots, \{n\}) = I$, it follows that $\Psi$ is an isomorphism of monoids. This concludes the proof.

\end{proof}

Now we are in a position to prove the main result of the paper.

\begin{proof}[Proof of Theorem~\ref{thm:main}]
   Theorem~\ref{thm:Phi} gives the isomorphism
   \[
      \Phi : \operatorname{End}(K_n) \rightarrow M_n.
   \]
   Theorem~\ref{thm:Psi} shows that
   \[
      \Psi : M_n \rightarrow D_n
   \]
   is also an isomorphism.  Composing these two maps yields an isomorphism
   \[
      \Psi \circ \Phi : \operatorname{End}(K_n) \rightarrow D_n,
   \]
   so \(\operatorname{End}(K_n) \cong (D_n,\cdot)\).  This completes the proof of the main theorem.  \qedhere
\end{proof}

\section{Concluding remarks}

Even though the following proposition was proved in \cite{GK} in Proposition 20, we can prove it again with relative ease.
\begin{proposition}
   Let $n\in\mathbb{N}$. The only invertible element of $\operatorname{End}(K_n)$ is the identity. Consequently, $\operatorname{Aut}(K_n)$ is trivial.
\end{proposition}

\begin{proof}
   In \cite{Rutherford} it is shown that the only invertible $n \times n$ Boolean matrices are the permutation matrices.
   Since the only permutation matrix that is an element of $D_n$ is the identity matrix, it follows that the only invertible element of $D_n$ is the identity matrix. This concludes the proof.
\end{proof}

Finding an explicit formula for the cardinality of $\operatorname{End}(K_n)$ or equivalently of $D_n$ has proven to be challenging. The authors in \cite{Kitaev} study matrices avoiding certain patterns. If we denote by $c_{m, n}$ the number of $m \times n$ Boolean matrices avoiding the pattern (submatrix) \eqref{pattern}, they proposed a finite algorithm which, for each fixed $m \geq 1$, gives a closed formula for $c_{m, n}$ for any $n \geq 1$. This result (Corollary 13 in \cite{Kitaev}) is summarised in the following proposition.
\begin{proposition}
   Let $n \in \mathbb{N}$. Then we have
   \begin{itemize}
      \item $c_{2, n} = (3 + n) \cdot 3^{n-1}$,
      \vspace{0.2cm}
      \item $c_{3, n} = \frac{1}{3} \cdot (2 + n)(96 + 31n + n^2) \cdot 4^{n - 3}$,
      \vspace{0.2cm}
      \item $c_{4, n} = \frac{1}{36} \cdot (2812500 + 3963450n + 1862971n^2 + 339300n^3 + 21265n^4 + 510n^5 + 4n^6) \cdot 5^{n-7}$.
   \end{itemize}
\end{proposition} 
From this, one can determine the cardinality of $D_n$ for $2 \leq n \leq 4$.

\backmatter

\bmhead{Acknowledgements}

I would like to thank Ganna Kudryavtseva for her valuable mentorship, insightful guidance, and helpful suggestions throughout the preparation of this paper.
I would also like to thank the anonymous referee for their careful reading and valuable suggestions.

%

\begin{thebibliography}{17}
\ifx \bisbn   \undefined \def \bisbn  #1{ISBN #1}\fi
\ifx \binits  \undefined \def \binits#1{#1}\fi
\ifx \bauthor  \undefined \def \bauthor#1{#1}\fi
\ifx \batitle  \undefined \def \batitle#1{#1}\fi
\ifx \bjtitle  \undefined \def \bjtitle#1{#1}\fi
\ifx \bvolume  \undefined \def \bvolume#1{\textbf{#1}}\fi
\ifx \byear  \undefined \def \byear#1{#1}\fi
\ifx \bissue  \undefined \def \bissue#1{#1}\fi
\ifx \bfpage  \undefined \def \bfpage#1{#1}\fi
\ifx \blpage  \undefined \def \blpage #1{#1}\fi
\ifx \burl  \undefined \def \burl#1{\textsf{#1}}\fi
\ifx \doiurl  \undefined \def \doiurl#1{\url{https://doi.org/#1}}\fi
\ifx \betal  \undefined \def \betal{\textit{et al.}}\fi
\ifx \binstitute  \undefined \def \binstitute#1{#1}\fi
\ifx \binstitutionaled  \undefined \def \binstitutionaled#1{#1}\fi
\ifx \bctitle  \undefined \def \bctitle#1{#1}\fi
\ifx \beditor  \undefined \def \beditor#1{#1}\fi
\ifx \bpublisher  \undefined \def \bpublisher#1{#1}\fi
\ifx \bbtitle  \undefined \def \bbtitle#1{#1}\fi
\ifx \bedition  \undefined \def \bedition#1{#1}\fi
\ifx \bseriesno  \undefined \def \bseriesno#1{#1}\fi
\ifx \blocation  \undefined \def \blocation#1{#1}\fi
\ifx \bsertitle  \undefined \def \bsertitle#1{#1}\fi
\ifx \bsnm \undefined \def \bsnm#1{#1}\fi
\ifx \bsuffix \undefined \def \bsuffix#1{#1}\fi
\ifx \bparticle \undefined \def \bparticle#1{#1}\fi
\ifx \barticle \undefined \def \barticle#1{#1}\fi
\bibcommenthead
\ifx \bconfdate \undefined \def \bconfdate #1{#1}\fi
\ifx \botherref \undefined \def \botherref #1{#1}\fi
\ifx \url \undefined \def \url#1{\textsf{#1}}\fi
\ifx \bchapter \undefined \def \bchapter#1{#1}\fi
\ifx \bbook \undefined \def \bbook#1{#1}\fi
\ifx \bcomment \undefined \def \bcomment#1{#1}\fi
\ifx \oauthor \undefined \def \oauthor#1{#1}\fi
\ifx \citeauthoryear \undefined \def \citeauthoryear#1{#1}\fi
\ifx \endbibitem  \undefined \def \endbibitem {}\fi
\ifx \bconflocation  \undefined \def \bconflocation#1{#1}\fi
\ifx \arxivurl  \undefined \def \arxivurl#1{\textsf{#1}}\fi
\csname PreBibitemsHook\endcsname

\bibitem[\protect\citeauthoryear{Aragona and D’Andrea}{2013}]{aragona13}
\begin{barticle}
\bauthor{\bsnm{Aragona}, \binits{R.}},
\bauthor{\bsnm{D’Andrea}, \binits{A.}}:
\batitle{{Hecke-Kiselman} monoids of small cardinality}.
\bjtitle{Semigr. Forum}
\bvolume{86}(\bissue{1}),
\bfpage{32}--\blpage{40}
(\byear{2013})
\end{barticle}
\endbibitem

\bibitem[\protect\citeauthoryear{Ashikhmin et~al.}{2015}]{Ashikhmin15}
\begin{barticle}
\bauthor{\bsnm{Ashikhmin}, \binits{D.}},
\bauthor{\bsnm{Volkov}, \binits{M.}},
\bauthor{\bsnm{Zhang}, \binits{W.}}:
\batitle{The {Finite} {Basis} {Problem} for {Kiselman} {Monoids}}.
\bjtitle{Demonstr. Math.}
\bvolume{48}(\bissue{4}),
\bfpage{475}--\blpage{492}
(\byear{2015})
\end{barticle}
\endbibitem

\bibitem[\protect\citeauthoryear{Collina and D’Andrea}{2015}]{Graphs}
\begin{barticle}
\bauthor{\bsnm{Collina}, \binits{E.}},
\bauthor{\bsnm{D’Andrea}, \binits{A.}}:
\batitle{A graph-dynamical interpretation of {Kiselman’s} semigroups}.
\bjtitle{J. Algebr. Comb.}
\bvolume{41}(\bissue{4}),
\bfpage{1115}--\blpage{1132}
(\byear{2015})
\end{barticle}
\endbibitem

\bibitem[\protect\citeauthoryear{D'Andrea and Stella}{2023}]{dandrea23}
\begin{barticle}
\bauthor{\bsnm{D'Andrea}, \binits{A.}},
\bauthor{\bsnm{Stella}, \binits{S.}}:
\batitle{The cardinality of {Kiselman's} semigroups grows double-exponentially}.
\bjtitle{Bull. Belg. Math. Soc. Simon Stevin}
\bvolume{30}(\bissue{5}),
\bfpage{570}--\blpage{576}
(\byear{2023})
\end{barticle}
\endbibitem

\bibitem[\protect\citeauthoryear{Ganyushkin and Mazorchuk}{2011}]{ganyushkin11}
\begin{barticle}
\bauthor{\bsnm{Ganyushkin}, \binits{O.}},
\bauthor{\bsnm{Mazorchuk}, \binits{V.}}:
\batitle{On {Kiselman} quotients of 0-{Hecke} monoids}.
\bjtitle{Int. Electron. J. Algebra}
\bvolume{10}(\bissue{10}),
\bfpage{174}--\blpage{191}
(\byear{2011})
\end{barticle}
\endbibitem

\bibitem[\protect\citeauthoryear{Kiselman}{2002}]{Kiselman}
\begin{barticle}
\bauthor{\bsnm{Kiselman}, \binits{C.}}:
\batitle{A semigroup of operators in convexity theory}.
\bjtitle{Trans. Am. Math. Soc.}
\bvolume{354}(\bissue{5}),
\bfpage{2035}--\blpage{2053}
(\byear{2002})
\end{barticle}
\endbibitem

\bibitem[\protect\citeauthoryear{Kitaev et~al.}{2005}]{Kitaev}
\begin{barticle}
\bauthor{\bsnm{Kitaev}, \binits{S.}},
\bauthor{\bsnm{Mansour}, \binits{T.}},
\bauthor{\bsnm{Vella}, \binits{A.}}:
\batitle{Pattern avoidance in matrices}.
\bjtitle{J. Integer Seq.}
\bvolume{8}(\bissue{2}),
\bfpage{05}--\blpage{2216052216}
(\byear{2005})
\end{barticle}
\endbibitem


\bibitem[\protect\citeauthoryear{Kudryavtseva and Mazorchuk}{2009}]{GK}
\begin{barticle}
\bauthor{\bsnm{Kudryavtseva}, \binits{G.}},
\bauthor{\bsnm{Mazorchuk}, \binits{V.}}:
\batitle{On {Kiselman’s} semigroup}.
\bjtitle{Yokohama Math. J.}
\bvolume{55}(\bissue{1}),
\bfpage{22}--\blpage{46}
(\byear{2009})
\end{barticle}
\endbibitem

\bibitem[\protect\citeauthoryear{Lebed}{2025}]{Lebed2025}
\begin{barticle}
\bauthor{\bsnm{Lebed}, \binits{V.}}:
\batitle{The word problem for {Hecke--Kiselman} monoids of type {$A_n$} and {$\widetilde A_n$}}.
\bjtitle{Semigr. Forum}
\bvolume{110}(\bissue{3}),
\bfpage{597}--\blpage{614}
(\byear{2025})
\end{barticle}
\endbibitem

\bibitem[\protect\citeauthoryear{M{\k{e}}cel and Okni{\'{n}}ski}{2019}]{Mecel2019}
\begin{barticle}
\bauthor{\bsnm{M{\k{e}}cel}, \binits{A.}},
\bauthor{\bsnm{Okni{\'{n}}ski}, \binits{J.}}:
\batitle{Gr{\"o}bner basis and the automaton property of {Hecke--Kiselman} algebras}.
\bjtitle{Semigr. Forum}
\bvolume{99}(\bissue{2}),
\bfpage{447}--\blpage{464}
(\byear{2019})
\end{barticle}
\endbibitem

\bibitem[\protect\citeauthoryear{Norton}{1979}]{norton79}
\begin{barticle}
\bauthor{\bsnm{Norton}, \binits{P.}}:
\batitle{0-{Hecke} algebras}.
\bjtitle{J. Aust. Math. Soc.}
\bvolume{27}(\bissue{3}),
\bfpage{337}--\blpage{357}
(\byear{1979})
\end{barticle}
\endbibitem

\bibitem[\protect\citeauthoryear{Okni{\'n}ski and Wiertel}{2020}]{okninski20}
\begin{barticle}
\bauthor{\bsnm{Okni{\'n}ski}, \binits{J.}},
\bauthor{\bsnm{Wiertel}, \binits{M.}}:
\batitle{Combinatorics and structure of {Hecke--Kiselman} algebras}.
\bjtitle{Commun. Contemp. Math.}
\bvolume{22}(\bissue{07}),
\bfpage{2050022}
(\byear{2020})
\end{barticle}
\endbibitem

\bibitem[\protect\citeauthoryear{Okni{\'n}ski and Wiertel}{2021}]{okninski21}
\begin{barticle}
\bauthor{\bsnm{Okni{\'n}ski}, \binits{J.}},
\bauthor{\bsnm{Wiertel}, \binits{M.}}:
\batitle{On the radical of a {Hecke--Kiselman} algebra}.
\bjtitle{Algebras Represent. Theory}
\bvolume{24}(\bissue{6}),
\bfpage{1431}--\blpage{1440}
(\byear{2021})
\end{barticle}
\endbibitem

\bibitem[\protect\citeauthoryear{Rutherford}{1963}]{Rutherford}
\begin{barticle}
\bauthor{\bsnm{Rutherford}, \binits{D.E.}}:
\batitle{Inverses of {Boolean} matrices}.
\bjtitle{Glasg. Math. J.}
\bvolume{6}(\bissue{1}),
\bfpage{49}--\blpage{53}
(\byear{1963})
\end{barticle}
\endbibitem

\bibitem[\protect\citeauthoryear{Wiertel}{2022}]{wiertel22}
\begin{barticle}
\bauthor{\bsnm{Wiertel}, \binits{M.}}:
\batitle{Irreducible representations of {Hecke--Kiselman} monoids}.
\bjtitle{Linear Algebra Its Appl.}
\bvolume{640},
\bfpage{12}--\blpage{33}
(\byear{2022})
\end{barticle}
\endbibitem

\bibitem[\protect\citeauthoryear{Wiertel}{2023}]{wiertel23}
\begin{barticle}
\bauthor{\bsnm{Wiertel}, \binits{M.}}:
\batitle{The {Gelfand--Kirillov} dimension of {Hecke--Kiselman} algebras}.
\bjtitle{Forum Math.}
\bvolume{35}(\bissue{2}),
\bfpage{523}--\blpage{534}
(\byear{2023})
\end{barticle}
\endbibitem

\end{thebibliography}


\end{document}